\documentclass[a4paper,11pt]{article}
\usepackage[english]{babel}
\usepackage{mathrsfs}
\usepackage{makeidx}
\usepackage{amsfonts}
\usepackage{amsmath}
\usepackage{amsthm}
\usepackage{enumitem}
\usepackage{latexsym,amsmath,amssymb}
\usepackage{graphicx,color}
\usepackage{comment}
 \newcommand{\lvec}[1]{\overleftarrow{#1}}

\def\lcf{\lbrack\! \lbrack}
\def\rcf{\rbrack\! \rbrack}

\setenumerate[1]{label=(\roman*), ref=(\roman*)}
\usepackage[all]{xy}

\newtheorem{thm}{Theorem}[section]
\newtheorem{prop}[thm]{Proposition}
\newtheorem{definition}[thm]{Definition}
\newtheorem{lemma}[thm]{Lemma}

\newtheoremstyle{obs}
  {3pt}
  {3pt}
  {}
  {}
  {\bfseries}
  {.}
  {.5em}
  {}
\theoremstyle{obs}
\newtheorem{remark}[thm]{Remark}

\newtheorem{ex}[thm]{Example}

\newcommand{\R}{\mathbb{R}}      

\def\qed{\ifvmode\removelastskip\fi
{\unskip\nobreak\hfil\penalty50\hbox{}\nobreak\hfil \hbox{\vrule
height1.2ex width1.2ex}\parfillskip=0pt \finalhyphendemerits=0
\par \smallskip}}

\title{Inverse problem for Lagrangian systems on Lie algebroids and applications to  reduction by symmetries}

\author{\textsc{Mar\'ia Barbero-Li\~n\'an}\thanks{mbarbero@math.uc3m.es}\\
\small
Departamento de Matem\'aticas, Universidad Carlos III de Madrid, \\ \small Avenida de la Universidad 30, 28911 Legan\'es, Madrid, Spain
\\ \small and Instituto de Ciencias Matem\'aticas (CSIC-UAM-UC3M-UCM)\\ \and
\textsc{Marta Farr\'e Puiggal\'i \thanks{marta.farre@icmat.es},}
\textsc{David Mart\'{\i}n de Diego}\thanks{david.martin@icmat.es} \\
\small
Instituto de Ciencias Matem\'aticas (CSIC-UAM-UC3M-UCM) \\ \small  C/Nicol\'as
Cabrera 13-15, 28049 Madrid, Spain
}
\makeindex
\begin{document}

\maketitle

\begin{abstract}
The language of Lagrangian submanifolds is used to extend a geometric characterization of the inverse problem 
of the calculus of variations on tangent bundles to regular Lie algebroids. Since not all closed sections are locally exact on Lie algebroids, the Helmholtz conditions on Lie algebroids 
are necessary but not sufficient, so they give a weaker definition of the
inverse problem. As an application the Helmholtz conditions on Atiyah algebroids are obtained so that the
relationship between the inverse problem and the reduced inverse problem by symmetries can be described. Some examples and comparison with previous 
approaches in the literature are provided. 

\vspace{3mm}


\vspace{3mm}

\end{abstract}

\newpage
\tableofcontents
\newpage

\section{Introduction}

The inverse problem of the calculus of variations consists in determining if the solutions of a given system
of second order differential equations correspond with the solutions of the Euler-Lagrange equations for some regular Lagrangian.
 This problem in the general version remains unsolved. In~\cite{BFM} we contribute to it with a novel description
in terms of Lagrangian submanifolds of a symplectic manifold, also valid to study the constrained inverse problem by using
 isotropic submanifolds instead of Lagrangian ones. Here we extend the use of Lagrangian submanifolds to geometrically characterize 
the inverse problem on regular Lie algebroids, giving a different approach from~\cite{POP} and extending 
the results for Lie algebras described in~\cite{CM2008}. 

On Lie algebroids the role of the SODE (second order differential equation) is played by a SODE section \cite{LMM,Martinez}. Locally, $\ddot{x}^{i}=\Gamma^{i}(x,\dot{x})$ is replaced
by $\dot{x}^{i}=\rho^{i}_{\alpha}(x)y^{\alpha}, \dot{y}^{\alpha}=\Gamma^{\alpha}(x,y)$, where $(x^i,y^\alpha)$ are local coordinates on 
a Lie algebroid $E$. The inverse problem on Lie algebroids poses the same question as in the classical inverse problem~\cite{HIRSCH,Sonin1886,Helmholtz1887}: 
When is the above system equivalent to the Euler-Lagrange equations for some regular
 Lagrangian? More precisely, when is it possible to find a nondegenerate matrix of multipliers $g_{\alpha\beta}(x,y)$ such that
\begin{equation*} 
g_{\alpha\beta}(\dot{y}^{\alpha}-\Gamma^{\alpha})=\frac{d}{dt}\left( \frac{\partial L}{\partial y^{\beta}} \right)-\rho^{i}_{\beta}\frac{\partial L}{\partial x^{i}}+C^{\gamma}_{\beta\nu}y^{\nu}\frac{\partial L}{\partial y^{\gamma}}
\end{equation*}
has a regular solution $L$? 

This problem for a system of second order ordinary differential equations $\ddot{x}^{i}=\Gamma^{i}(x,\dot{x})$ has an 
extensive literature. The question was first raised by Hirsch in 1898 \cite{HIRSCH} and the main contribution 
to the problem was made by Douglas in \cite{Douglas} for the $2$-dimensional case giving an almost complete classification by studying a set of necessary and sufficient conditions in terms of the multipliers. 
A geometric interpretation of this set of necessary and sufficient conditions, so called Helmholtz conditions, was given 
in~\cite{94CSMBP} in terms of the Jacobi endomorphism, the dynamical covariant derivative and the vertical covariant derivative. Some particular cases in Douglas' classification 
have been generalized to arbitrary dimension but no classification close to the one given by Douglas is known in any 
dimension higher than $2$.

Note that the inverse problem of the calculus of variations also refers to the problem of matching a system of second order 
differential equations with the Euler-Lagrange equations for a regular Lagrangian in such a way that the matrix of multipliers is
equal to the identity matrix. This problem was posed earlier by Helmholtz in 1887 
who provided a set of necessary and sufficient conditions \cite{Helmholtz1887}. 

In the nineties of the last century two important contributions show how Lie algebroids and Lie groupoids~\cite{MAC} are very useful to 
describe Lagrangian mechanics~\cite{96Liber,WEINST}. From then on, the benefits of Lie algebroids to describe Lagrangian 
and Hamiltonian mechanics have become very clear in the literature~\cite{LMM,Martinez} and references therein. For instance, using the Atiyah algebroid framework,  Lagrange-Poincar\'e
and Hamilton-Poincar\'e equations are obtained very naturally~\cite{2007IMMP}.

The paper is organized as follows. In section \ref{background} we give the necessary 
background on the theory of Lie algebroids, including prolongations of Lie algebroids, the Tulczyjew isomorphism and symplectic Lie algebroids. In section \ref{sections} we discuss the lack of the Poincar\'{e} lemma 
for the differential associated to  general Lie algebroids and give a characterization of locally exact sections of the dual of a regular Lie algebroid. This 
is a key lemma in section \ref{inverse}. In section \ref{mechanics} we review the derivation of the Euler-Lagrange equations on 
a Lie algebroid in the way given in \cite{Martinez}. In section \ref{inverse} we identify the insufficiency of the 
Helmholtz conditions as the lack of the Poincar\'{e} lemma and give a characterization of the variationality of a SODE on a 
regular Lie algebroid using Lemma \ref{lemma-sec} in section \ref{sections}. We also give a generalization to SODEs
 on regular Lie algebroids of Crampin's characterization for SODEs on tangent bundles \cite{81Crampin} weakening
 the notion of variationality and we include an example of a SODE on a Lie algebroid that is not variational but satisfies the Helmholtz conditions. In section~\ref{Sec:Morph} we study how morphisms of Lie algebroids treat the variational condition for SODE sections. This generalizes with an intrinsic proof results in~\cite{CM2008} about the inverse problem on a Lie group and the corresponding reduced inverse problem on the Lie algebra. An interesting application appears in section~\ref{Sec:Atiyah} where the inverse problem on Atiyah algebroids is considered. In appendix \ref{approaches} we give the equivalence between the Helmholtz conditions derived in this paper and the Helmholtz conditions given in \cite{CM2008} for Lie algebras and in \cite{POP} for Lie algebroids. Note that in this last paper the insufficiency of the Helmholtz conditions is not discussed.

\section{Lie algebroids} \label{background}
In this section we give the background on the theory of Lie algebroids that will be needed later on. This includes 
prolongations of Lie algebroids, the Tulczyjew isomorphism for Lie algebroids, symplectic Lie algebroids and Lagrangian 
submanifolds. For further details we refer the reader to \cite{LMM,MAC} and references therein.

\begin{definition}
A Lie algebroid is a vector bundle $\tau:E\longrightarrow M$ together with a morphism $\rho:E\longrightarrow TM$ of vector bundles (called the anchor) and a Lie bracket in $\Gamma (E)$, the $\mathcal{C}^{\infty}(M)$-module of sections of $E$, satisfying the Leibniz rule
$$
[X,fY]=\rho(X)(f)Y+f[X,Y] \quad \mbox{ for all } X,Y\in\Gamma(E), \mbox{ and } f\in C^{\infty}(M).
$$
\end{definition}

Note that this is in particular a generalization of tangent bundles and Lie algebras.

Let $(x^{i})$ denote local coordinates on $M$ and $\left\{ e_{1},\ldots, e_{n} \right\}$ be a basis of local sections of $E$. With respect to this basis, the structure functions $\rho^{i}_{\alpha}$ and $C^{\gamma}_{\alpha\beta}$ of the Lie algebroid are functions on $M$ defined by
\begin{equation} \label{structure}
\rho(e_{\alpha})=\rho^{i}_{\alpha}\frac{\partial}{\partial x^{i}} \quad \mbox{ and } \quad [e_{\alpha},e_{\beta}]=C^{\gamma}_{\alpha\beta}e_{\gamma}.
\end{equation}
Since the anchor $\rho$ is an algebra morphism, that is $[\rho(e_{\alpha}),\rho(e_{\beta})]=\rho[e_{\alpha},e_{\beta}]$, and the Jacobi identity $\sum_{(\alpha,\beta,\gamma)} [e_{\gamma},[e_{\alpha},e_{\beta}]]=0$ holds, the structure functions must satisfy the structure equations
$$
\rho^{j}_{\alpha}\frac{\partial\rho^{i}_{\beta}}{\partial x^{j}}-\rho^{j}_{\beta}\frac{\partial \rho^{i}_{\alpha}}{\partial x^{j}}=\rho^{i}_{\gamma}C^{\gamma}_{\alpha\beta} \quad \mbox{ and } \quad \sum_{(\alpha,\beta,\gamma)}\left[ \rho^{i}_{\alpha}\frac{\partial C^{\nu}_{\beta\gamma}}{\partial x^{i}}+C^{\nu}_{\alpha\mu}C^{\mu}_{\beta\gamma} \right]=0.
$$

A Lie algebroid structure in a vector bundle $\tau:E\longrightarrow M$ is equivalent to an exterior differential $d^{E}$ in the dual vector bundle $\tau^{*}:E^{*}\longrightarrow M$, that is, an operator
$
d^{E}:\Gamma(\wedge^{*}E^{*})\longrightarrow \Gamma(\wedge^{*+1}E^{*})
$
satisfying
\begin{eqnarray*}
&&d^{E}\circ d^{E}=0,\\
&&d^{E}(\alpha\wedge\beta)=d^{E}\alpha\wedge\beta+(-1)^{\deg(\alpha)}\alpha\wedge d^{E}\beta,
\end{eqnarray*}
where $\alpha,\beta\in\Gamma(\wedge^{*}E^{*})$ and ${\deg(\alpha)}$ denotes the degree of $\alpha$.

If $\alpha\in\Gamma(\wedge^{n}E^{*})$, the exterior differential $d^{E}\alpha$ is defined from the bracket and the anchor map by
\begin{eqnarray*}
d^{E}\alpha(e_{0},\ldots,e_{n})&=&\sum_{i=0}^{n}(-1)^{i}\rho(e_{i})\alpha(e_{0},\ldots,\widehat{e_{i}},\ldots,e_{n})\\
&& +\sum_{i<j}(-1)^{i+j}\alpha([e_{i},e_{j}],e_{0},\ldots,\widehat{e_{i}},\ldots,\widehat{e_{j}},\ldots,e_{n}),
\end{eqnarray*}
where $e_{0}\ldots,e_{n}\in\Gamma(E).$ On the other hand, given an exterior differential $d^{E}$, the equations
$$
\rho(e)(f)=\langle d^{E}f,e \rangle \mbox{ and } \langle \alpha,[e_{1},e_{2}] \rangle = \rho(e_{1})\langle \alpha, e_{2} \rangle - \rho(e_{2})\langle \alpha , e_{1} \rangle - d^{E}\alpha(e_{1},e_{2}),
$$
where $\alpha\in\Gamma(E^{*})$, $e,e_{1},e_{2}\in\Gamma(E)$ and $f\in C^{\infty}(M)$, define an anchor $\rho$ and a Lie bracket of sections $[\cdot,\cdot]$ for $E$.

If $\left\{e^{\alpha}\right\}$ denotes the dual basis to $\left\{e_{\alpha}\right\}$ then for $f\in\mathcal{C}^{\infty}(M)$ and $\theta=\theta_{\alpha}e^{\alpha}\in\Gamma(E^{*})$ the local expressions of the differentials are
$$
d^{E}f=\frac{\partial f}{\partial x^{i}}\rho^{i}_{\alpha}e^{\alpha} \quad \mbox{ and } \quad d^{E}\theta=\left(\frac{\partial\theta_{\gamma}}{\partial x^{i}}\rho^{i}_{\beta}-\frac{1}{2}\theta_{\alpha}C^{\alpha}_{\beta\gamma}\right)e^{\beta}\wedge e^{\gamma}.
$$

The following definitions will be used in order to introduce Lagrangian submanifolds on Lie algebroids.

\begin{definition}
A Lie algebroid morphism is a morphism of vector bundles $F:E\longrightarrow E'$ over $f:M\longrightarrow M'$ such that $d^{E}((F,f)^{*}\phi')=(F,f)^{*}(d^{E^{'}}\phi')$, for all $\phi'\in\Gamma(\bigwedge^{k}(E')^{*})$.
A Lie algebroid epimorphism is a Lie algebroid morphism $(F,f)$ such that $f$ is a surjective submersion and 
$F|_{E_x}\colon E_x\rightarrow E'_{f(x)}$ is a linear epimorphism for all $x\in M$.
\end{definition} 

\begin{definition}
A Lie subalgebroid is a morphism of Lie algebroids $j : F \longrightarrow E$, $i : N\longrightarrow M$ such that the pair $(j, i)$ is a monomorphism of vector bundles and $i$ is an injective immersion.
\end{definition}

\subsection{Prolongations of Lie algebroids}
Now we will introduce the prolongation of a Lie algebroid over a smooth map $f:M'\longrightarrow M$. This notion will allow a 
derivation of the Euler-Lagrange equations without using the Poisson bracket on the dual of the Lie algebroid. This is the analog 
of the Klein formalism for tangent bundles \cite{Klein}; it was given by Mart\'{i}nez in \cite{Martinez} for Lie algebroids and will be recalled in the next section.

In order to guarantee that the following construction is a vector bundle, a constant $c$ is needed such that
\begin{equation}
\mbox{dim}\left(\rho(E_{f(x')})+(T_{x'}f)(T_{x'}M') \right)=c \quad \mbox{ for all } x'\in M'. \label{dim}
\end{equation}
This condition implies that the dimension of the fibers must be constant.

\begin{definition}[\cite{LMM},\cite{Higgins}]
Let $(E,[\cdot,\cdot],\rho)$ be a Lie algebroid over a manifold $M$ with projection denoted by $\tau$, and $f:M'\longrightarrow M$ a smooth map satisfying (\ref{dim}). Then the prolongation of $E$ over $f$ is the Lie algebroid $(\mathcal{L}^{f}E,[\cdot,\cdot]^{f},\rho^{f})$ over $M'$ with total space
$$
\mathcal{L}^{f}E=\left\{ (b,v') \in E\times TM' : \rho(b)=(Tf)(v') \right\}
$$
and projection $\tau^{f}:\mathcal{L}^{f}E\longrightarrow M'$ given by $\tau^{f}(b,v')=\tau_{M'}(v')$. The sections of $\mathcal{L}^{f}E$ are of the form $(h^{i}(X_{i}\circ f),X')$, where $X'\in\mathfrak{X}(M')$, $X_{i}\in\Gamma(E)$ and $h^{i}\in\mathcal{C}^{\infty}(M')$. Then the Lie bracket is defined by
$$
[(h^{i}(X_{i}\circ f),X'),(s^{j}(Y_{j}\circ f),Y')]^{f}=(h^{i}s^{j}([X_{i},Y_{j}]\circ f)+X'(s^{j})(Y_{j}\circ f)-Y'(h^{i})(X_{i}\circ f) , [X',Y'])
$$
where $X',Y'\in\mathfrak{X}(M')$, $X_{i},Y_{i}\in\Gamma(E)$ and $h^{i},s^{i}\in\mathcal{C}^{\infty}(M')$. Note that the bracket in the second factor denotes the usual bracket of vector fields. Finally the anchor is given by the projection onto the second factor:
$$ 
\begin{array}{cccc}
\rho^{f}: & \mathcal{L}^{f}E & \longrightarrow & TM' \\
& (b,v') & \longmapsto & v'.
\end{array}
$$
\end{definition}

In particular if we take $f$ to be the projections $\tau:E\longrightarrow M$ and $\tau^{*}:E^{*}\longrightarrow M$ respectively then the prolongations $\mathcal{L}^{\tau}E$ and $\mathcal{L}^{\tau^{*}}E$ play the roles of $TTQ$ and $TT^{*}Q$ respectively, which are recovered when $E=TQ$. These are the prolongations that we will use in this paper, so we introduce now local coordinates. 

Let $\left\{ e_{\alpha} \right\}$ denote a basis of local sections of $\tau:E\longrightarrow M$ and $(x^{i},y^{\alpha})$ the corresponding coordinates on $E$. Having in mind the structure functions defined in (\ref{structure}), we consider the basis of local sections of $\mathcal{L}^{\tau}E\longrightarrow E$ given by
\begin{equation} \label{basis}
\tilde{T}_{\alpha}(a)=\left( e_{\alpha}(\tau(a)), \rho^{i}_{\alpha}\left.\frac{\partial}{\partial x^{i}}\right|_{a} \right), \quad \tilde{V}_{\alpha}(a)=\left( 0, \left.\frac{\partial}{\partial y^{\alpha}}\right|_{a} \right), \quad a\in E
\end{equation}
following the notation in \cite{LMM}.

With respect to this basis the structure functions are given by
$$
[\tilde{T}_{\alpha},\tilde{T}_{\beta}]^{\tau}=C^{\gamma}_{\alpha\beta}\tilde{T}_{\gamma}, \quad [\tilde{T}_{\alpha},\tilde{V}_{\beta}]^{\tau}=0, \quad [\tilde{V}_{\alpha},\tilde{V}_{\beta}]^{\tau}=0,
$$
$$
\rho^{\tau}(\tilde{T}_{\alpha})=\rho^{i}_{\alpha}\frac{\partial}{\partial x^{i}}, \quad \rho^{\tau}(\tilde{V}_{\alpha})=\frac{\partial}{\partial y^{\alpha}},
$$
and the local coordinates induced on $\mathcal{L}^{\tau}E$ will be denoted by $(x^{i},y^{\alpha},z^{\alpha},v^{\alpha})$.

Let $\left\{ e^{\alpha} \right\}$ be the dual basis to $\left\{ e_{\alpha} \right\}$ and $(x^{i},y_{\alpha})$ the corresponding coordinates on $E^{*}$. We consider the basis of local sections of $\mathcal{L}^{\tau^{*}}E\longrightarrow E^{*}$
$$
\tilde{T}_{\alpha}(a^{*})=\left( e_{\alpha}(\tau^{*}(a^{*})), \rho^{i}_{\alpha}\left.\frac{\partial}{\partial x^{i}}\right|_{a^{*}} \right), \quad \tilde{V}^{\alpha}(a^{*})=\left( 0, \left.\frac{\partial}{\partial y_{\alpha}}\right|_{a^{*}} \right), \quad a^{*}\in E^{*}
$$
with structure functions given by
$$
[\tilde{T}_{\alpha},\tilde{T}_{\beta}]^{\tau^{*}}=C^{\gamma}_{\alpha\beta}\tilde{T}_{\gamma}, \quad [\tilde{T}_{\alpha},\tilde{V}^{\beta}]^{\tau^{*}}=0, \quad [\tilde{V}^{\alpha},\tilde{V}^{\beta}]^{\tau^{*}}=0,
$$
$$
\rho^{\tau^{*}}(\tilde{T}_{\alpha})=\rho^{i}_{\alpha}\frac{\partial}{\partial x^{i}}, \quad \rho^{\tau^{*}}(\tilde{V}_{\alpha})=\frac{\partial}{\partial y_{\alpha}}.
$$
The local coordinates induced on $\mathcal{L}^{\tau^{*}}E\longrightarrow E^{*}$ will be denoted by $(x^{i},y_{\alpha},z^{\alpha},v_{\alpha})$.

\begin{remark}
A map $F:E\longrightarrow E^{*}$ over $M$ induces a map $\mathcal{L}F:\mathcal{L}^{\tau}E\longrightarrow \mathcal{L}^{\tau^{*}}E$ defined by
$$
\mathcal{L}F(b,X_{a}):=(b,T_{a}F(X_{a})).
$$
If locally $F(x^{i},y^{\alpha})=(x^{i},F_{\alpha}(x,y))$, then the local expression for $\mathcal{L}F$ is
$$
\mathcal{L}F(x^{i},y^{\alpha},z^{\alpha},v^{\alpha})=\left(x^{i},F_{\alpha},z^{\alpha},\rho^{i}_{\beta}z^{\beta}\frac{\partial F_{\alpha}}{\partial x^{i}}+v^{\beta}\frac{\partial F_{\alpha}}{\partial y^{\beta}} \right).
$$
\end{remark}

\subsection{Lagrangian submanifolds of symplectic Lie algebroids}

According to the philosophy in~\cite{BFM}, we must define Lagrangian submanifolds of symplectic Lie algebroids, see~\cite{LMM} 
for more details. 

\begin{definition}
A symplectic section $\Omega$ on a Lie algebroid $(E,[\cdot, \cdot],\rho)$ is a closed section of the vector bundle 
$E^{*}\wedge E^{*}\longrightarrow M$ satisfying that $\Omega_{x}:E_{x}\wedge E_{x}\longrightarrow \mathbb{R}$ is
 non-degenerate, that is, each fiber is a symplectic vector space. A Lie algebroid with a symplectic section will be 
called a symplectic Lie algebroid. 
\end{definition}

\begin{ex}
The Lie algebroid $\mathcal{L}^{\tau^{*}}E$ has a canonical symplectic section defined as $\Omega_{E}=-d^{\mathcal{L}^{\tau^{*}}E}\lambda_{E}$, where $\lambda_{E}$ is the canonical section of $(\mathcal{L}^{\tau^{*}}E)^{*}\longrightarrow E^{*}$ given by
$$
\lambda_{E}(a^{*})(b,v)=a^{*}(b) \quad \mbox{ for all } a^{*}\in E^{*}
$$
and is called the Liouville section.
\end{ex}

Once a symplectic section has been defined, we can introduce Lagrangian submanifolds of the Lie algebroid. 
As mentioned before, we are interested in Lagrangian submanifolds to extend the geometric setting of the inverse problem in~\cite{BFM}
to Lie algebroids. 

\begin{definition}
Let $\Omega$ be a symplectic section on $E$. The Lie subalgebroid $j:F\longrightarrow E$, $i:N\longrightarrow M$ is 
called Lagrangian if $j(F_{x})$ is a Lagrangian subspace of $(E_{i(x)},\Omega_{i(x)})$ for each $x\in N$.
\end{definition}

The Tulczyjew isomorphism in classical mechanics can be extended to the Lie algebroid setting. In this context the canonical
isomorphism is between $\rho^{*}(TE^{*})$ and $(\mathcal{L}^{\tau}E)^{*}$
$$
\xymatrix{
& \mathcal{L}^{\tau^{*}}E \equiv \rho^{*}(TE^{*}) \ar[dl] \ar[dr] \ar[rr]^{A_{E}} & & (\mathcal{L}^{\tau}E)^{*} \ar[dl] \\
E^{*} & & E &
}
$$
and is locally given by 
$A_{E}(x^{i},y_{\alpha},z^{\alpha},v_{\alpha})=(x^{i},z^{\alpha},v_{\alpha}+C^{\gamma}_{\alpha\beta}y_{\gamma}z^{\beta},y_{\alpha})$.
 For an intrinsic definition and more details we refer the reader to \cite{LMM}.

\begin{remark}
The vector bundles $\mathcal{L}^{\tau^{*}}E\longrightarrow E^{*}$ and $\rho^{*}(TE^{*})\longrightarrow E$ have the same total 
spaces but different projections.
\end{remark}

Now we will recall Proposition 7.8 in \cite{LMM} which will be used in the sequel. Let $\tau^{N}$ denote the projection $\tau:E\longrightarrow M$ restricted to a submanifold $i:N\hookrightarrow E$, that is, 
$\tau^{N}=\tau\circ i$. 

\begin{prop}
Given a section $\tilde{X}$ of the pull-back vector bundle $\rho^{*}(TE^{*})\longrightarrow E$ define $\alpha_{\tilde{X}}=A_{E}\circ\tilde{X}$, which is a section of $(\mathcal{L}^{\tau}E)^{*}\longrightarrow E$, and put $N=\tilde{X}(E)$. Then the Lie subalgebroid $(Id,Ti):\mathcal{L}^{(\tau^{\tau^{*}})^{N}}(\mathcal{L}^{\tau^{*}}E)\longrightarrow \mathcal{L}^{\tau^{\tau^{*}}}(\mathcal{L}^{\tau^{*}}E)$, $i:N \longrightarrow \mathcal{L}^{\tau^{*}}E$ is Lagrangian if and only if $d^{\mathcal{L}^{\tau}E}\alpha_{\tilde{X}}=0$, where $\mathcal{L}^{(\tau^{\tau^{*}})^{N}}(\mathcal{L}^{\tau^{*}}E)$ is the prolongation of $\mathcal{L}^{\tau^{*}}E$ over the map $(\tau^{\tau^{*}})^{N}:N\longrightarrow E^{*}$.
\end{prop}

\begin{remark}
According to Definition 8.1 in \cite{LMM}, $N=\tilde{X}(E)$ is a Lagrangian submanifold of $\mathcal{L}^{\tau^{*}}E$.
\end{remark}

\section{Closed sections versus exact sections} \label{sections}

On Lie algebroids the Poincar\'{e} lemma does not hold in general for the differential $d^{E}$, that is, the closedness of a section does not guarantee its local exactness.

\begin{ex}
Consider the Example 3.3.6 in \cite{MAC}, that is, the Lie algebroid with total space $E=T\mathbb{R}$, base space $M=\mathbb{R}$, Lie bracket defined by
$$
\left[\xi\frac{d}{dt},\eta\frac{d}{dt}\right]'=t\left(\frac{d\eta}{dt}\xi-\frac{d\xi}{dt}\eta\right)\frac{d}{dt}
$$
for functions $\xi,\eta:\mathbb{R}\longrightarrow\mathbb{R}$, where $t$ denotes the coordinate on $\mathbb{R}$, and anchor given by
$$
\begin{array}{lccc}
\rho: & T\mathbb{R} & \longrightarrow & T\mathbb{R} \\
& \xi\frac{d}{dt} & \longmapsto & t\xi\frac{d}{dt}\,.
\end{array}
$$
Thus, the structure functions are $\rho^{1}_{1}=t$ and $C^{1}_{11}=0$. Note that this algebroid is not regular since $\rho(E_{0})=0$ 
while rank$(\rho(E_{t}))=1$ for $t\not=0$, where $E_{t}$ denotes the fiber of $E$ over $t$ in $M$.

We want to detect a section of $T^{*}\mathbb{R}\longrightarrow\mathbb{R}$ which is closed but not locally exact. Note first that, 
by dimension, $d^{T\mathbb{R}}\theta=0$ for all $\theta=\alpha(t)dt\in\Gamma(T^{*}\mathbb{R})$. Since 
$d^{T\mathbb{R}}f=\frac{df}{dt}tdt$ for $f:\mathbb{R}\longrightarrow\mathbb{R}$, it suffices to take $\alpha(t)$ equal to
a nonzero constant $c$ so that the equation $\alpha(t)=t\frac{df}{dt}$ is not satisfied around $0$.
\end{ex}

We will give a characterization of the local exactness of a section of the dual of a regular Lie algebroid. For that we use some suitable coordinates given by the local splitting theorem in \cite{RUI}. If $E$ is regular, that is, $\rho$ has constant rank $q$, then the theorem reduces to the following one:

\begin{thm}[\cite{RUI}] \label{splitting}
Let $(E,[\cdot,\cdot],\rho)$ be a regular Lie algebroid over $M$ and let $x_{0}\in M$. There exist coordinates $(x^{i})$, $i=1,\ldots,m=\dim(M)$ in a neighborhood $U$ of $x_{0}$ and a basis of sections $\left\{e_{1},\ldots,e_{n}\right\}$ of $\tau^{-1}(U)\longrightarrow U$ such that
\begin{eqnarray*}
&&\rho(e_{i})=\frac{\partial}{\partial x^{i}}, \quad i=1,\ldots,q, \\
&&\rho(e_{s})=0, \quad s=q+1,\ldots,n.
\end{eqnarray*}
Moreover $C^{\alpha}_{\beta\gamma}=0$ for all $\alpha\leq q$.
\end{thm}

The characterization reads as follows:

\begin{lemma} \label{lemma-sec}
Let $(E,[\cdot,\cdot],\rho)$ be a regular Lie algebroid over $M$. A section $\theta$ of $\tau^{*}:E^{*}\longrightarrow M$ is locally exact if and only if it is closed and it satisfies $\theta(Z)=0$ for all $Z\in\Gamma(\mbox{Ker}(\rho))$.
\end{lemma}
\begin{proof}
$\Rightarrow$ Let $\left\{ e^{1},\ldots, e^{n} \right\}$ denote a local basis of $\tau^{*}:E^{*}\longrightarrow M$ and write $\theta=\theta_{\gamma}(x)e^{\gamma}$. If $\theta=d^{E}f$ locally,  then $d^{E}\theta=0$. The second condition also holds since $\theta_{\gamma}=\frac{\partial f}{\partial x^{i}}\rho^{i}_{\gamma}$ and then for each $X=X^{\gamma}e_{\gamma}\in\Gamma(\mbox{Ker}(\rho))$ we have
$$\theta(X)=\theta_{\gamma}X^{\gamma}=\frac{\partial f}{\partial x^{i}}\underbrace{\rho^{i}_{\gamma}X^{\gamma}}_{=0}=0.$$

$\Leftarrow$ To prove the converse result take the coordinates $(x^{i})$ on $M$ and the basis $\left\{e_{1},\ldots, e_{n}\right\}$ 
of sections of $E\longrightarrow M$ given in the splitting Theorem \ref{splitting}, so that $\left\{e_{q+1},\ldots, e_{n}\right\}$ 
is a basis of $\Gamma(\mbox{Ker}(\rho))$. Let $\left\{e^{1},\ldots,e^{n}\right\}$ denote the dual basis. If 
$\theta$ annihilates the sections $\Gamma(\mbox{Ker}(\rho))$, then it is written as $\theta=\theta_{\gamma}(x^{i})e^{\gamma}$ for $\gamma=1,\ldots,q$.

Locally, the condition $d^{E}\theta=0$ reads
$$
\frac{\partial \theta_{\gamma}}{\partial x^{i}}\rho^{i}_{\beta}-\frac{\partial \theta_{\beta}}{\partial x^{i}}\rho^{i}_{\gamma}-\theta_{\alpha}C^{\alpha}_{\beta\gamma}=0 \quad \mbox{ for all } \beta,\gamma=1,\ldots n,\quad i=1\ldots m.
$$
Using that $\rho^{i}_{\beta}=0$ for $\beta>q$, $\rho^{i}_{\beta}=\delta^{i}_{\beta}$ for $\beta\leq q$ and $C^{\alpha}_{\beta\gamma}=0$ for $\alpha\leq q$ in the chosen coordinates and also that $\theta_{\gamma}=0$ for $\gamma>q$ the above condition reduces to
$$
\frac{\partial \theta_{\gamma}}{\partial x^{\beta}}-\frac{\partial \theta_{\beta}}{\partial x^{\gamma}}=0, \quad \beta,\gamma=1,\ldots,q,
$$ 
which is precisely the integrability condition that provides locally a function $f(x)$ such that $\theta_{\gamma}=\frac{\partial f}{\partial x^{\gamma}}$, $\gamma=1\ldots q$.
\end{proof}

Next we give an example of a regular Lie algebroid for which the Poincar\'{e} lemma is not satisfied:
\begin{ex} \label{se2}
Consider the Lie algebra $E=\mathfrak{se}(2)$ with generators 
$$
e_{1} = \left(\begin{array}{ccc} 0&0&1\\ 0&0&0\\ 0&0&0 \end{array}\right),
 \quad e_{2} = \left(\begin{array}{ccc} 0&0&0\\ 0&0&1\\ 0&0&0 \end{array}\right),
 \quad e_{3} = \left(\begin{array}{ccc} 0&-1&0\\ 1&0&0\\ 0&0&0 \end{array}\right),
$$
Lie bracket given by
$$
[e_{1},e_{2}]=0, \quad [e_{1},e_{3}]=-e_{2} \quad \mbox{ and } \quad [e_{2},e_{3}]=e_{1}
$$
and anchor $\rho\equiv 0$. Let $\left\{e^{1},e^{2},e^{3}\right\}$ denote the dual basis. 
Note that $d^{E}(e^{3})$=0, since $C_{\alpha\beta}^{3}=0$. Note also that $\mbox{Ker}(\rho)=\left\{e_{1},e_{2},e_{3}\right\}$ 
and $e^{3}(e_{3})=1\not=0$, that is, the second condition in Lemma \ref{lemma-sec} is not satisfied and therefore $e^3$ is not locally exact.
\end{ex}

\section{The inverse problem of the calculus of variations on Lie algebroids}\label{S:Lie}

We first need to introduce briefly Lagrangian mechanics on Lie algebroids so that the geometric framework of the inverse problem on Lie
algebroids can be described. 

\subsection{Lagrangian mechanics on Lie algebroids} \label{mechanics}
We give a derivation of the Euler-Lagrange equations for a Lagrangian on a Lie algebroid following \cite{Martinez}. These equations were previously derived in \cite{WEINST} using the Poisson structure in the dual bundle.

The vertical endomorphism and the Liouville vector field on tangent bundles can be generalized to Lie algebroids. 
Note that these are the two ingredients needed to define the concept of a SODE section. First we give the definitions of the vertical and complete lifts of a section of $E\longrightarrow M$ to a section of $\mathcal{L}^{\tau}E\longrightarrow E$.

\begin{definition}
Let $X\in\Gamma(E)$. 
\begin{itemize}
\item The vertical lift of $X$ is the section $X^{v}\in\Gamma(\mathcal{L}^{\tau}E)$ defined by
$X^{v}(a)=(0,X(\tau(a))_{a}^{v}), a\in E$, where for each pair $a,b\in E$, $b_{a}^{v}$ acts on a function $F\in\mathcal{C}^{\infty}(E)$ as
$$
b_{a}^{v}(F)=\left.\frac{d}{dt}\right|_{t=0} F\left(a+tb\right).
$$

\item The complete lift of $X$ is the unique section $X^{c}\in\Gamma(\mathcal{L}^{\tau}E)$ that projects over $X$ and satisfies
$$
\rho^{\tau}(X^{c})\widehat{\theta}=\widehat{\mathcal{L}_{X}^{E}\theta} \mbox{ for all } \theta\in\Gamma(E^{*}), 
$$
where $\widehat{\theta}:E\longrightarrow \mathbb{R}$ is the linear function defined by the pairing $\widehat{\theta}(e)=\langle \theta(\tau^{*}(e)),e \rangle$ and $\mathcal{L}_{X}^{E}:=i_{X}\circ d^{E}+d^{E}\circ i_{X}$ is the Lie derivative.
\end{itemize}
\end{definition}

\begin{definition}
Given a Lie algebroid $E\longrightarrow M$,
\begin{itemize}
\item the vertical endomorphism $S$ is the unique section of $\mathcal{L}^{\tau}E\otimes\mathcal{L}^{\tau}E\longrightarrow E$ satisfying
$$
S(X^{v})=0, \quad S(X^{c})=X^{v} \mbox{ for all } X\in\Gamma(E),
$$

\item the Euler section $\Delta$ is the section of $\mathcal{L}^{\tau}E\longrightarrow E$ defined by
$$
\Delta(a)=(0,a_{a}^{v}) \mbox{ for all } a\in E.
$$
\end{itemize}
\end{definition}

\begin{definition}
A section $\Gamma$ of $\mathcal{L}^{\tau}E\longrightarrow E$ is a second order differential equation (SODE) if it satisfies $S(\Gamma)=\Delta$. We will use the expressions SODE section and SODE field to distinguish $\Gamma$ from $\rho^{\tau}(\Gamma)$.
\end{definition}

With respect to the basis $\{ \tilde{T}_{\alpha}, \tilde{V}_{\alpha} \}$ defined in (\ref{basis}), the local expression of a SODE section is
$$
\Gamma=y^{\alpha}\tilde{T}_{\alpha}+\Gamma^{\alpha}\tilde{V}_{\alpha}.
$$
As $\rho^{\tau}(\tilde{T}_{\alpha})=\rho^{i}_{\alpha}\frac{\partial}{\partial x^{i}}$ and $\rho^{\tau}(\tilde{V}_{\alpha})=\frac{\partial}{\partial y^{\alpha}}$, the local expression for the SODE field is
$$
\rho^{\tau}(\Gamma)=y^{\alpha}\rho^{i}_{\alpha}\frac{\partial}{\partial x^{i}}+\Gamma^{\alpha}\frac{\partial}{\partial y^{\alpha}},
$$
so the integral curves of $\rho^{\tau}(\Gamma)$ are the solutions to $\dot{x}^{i}=\rho^{i}_{\alpha}y^{\alpha}$ and $\dot{y}^{\alpha}=\Gamma^{\alpha}(x,y)$. 

If we also have a Lagrangian function $L:E\longrightarrow \mathbb{R}$ on the Lie algebroid, then we can define the Poincar\'{e}-Cartan $1$-form $\theta_{L}$ and $2$-form $\omega_{L}$ and the energy function $E_{L}$ as follows:
$$
\theta_{L}=S(d^{E}L), \quad \omega_{L}=-d^{E}\theta_{L}, \quad E_{L}=\rho^{\tau}(\Delta)(L)-L.
$$

If $L$ is regular, then $\omega_{L}$ is a symplectic section and the Hamiltonian equation
$$
i_{\Gamma}\omega_{L}=d^{E}E_{L}
$$
has a unique solution $\Gamma_{L}$. The integral curves of $\Gamma_{L}$ are the integral curves of $\rho^{\tau}(\Gamma_{L})$, which are those locally satisfying the Euler-Lagrange equations for a Lie algebroid:
\begin{eqnarray*}
&&\frac{dx^{i}}{dt}=\rho^{i}_{\alpha}y^{\alpha}, \\
&&\frac{d}{dt}\left( \frac{\partial L}{\partial y^{\alpha}} \right)=\rho^{i}_{\alpha}\frac{\partial L}{\partial x^{i}}-C^{\gamma}_{\alpha \beta}y^{\beta}\frac{\partial L}{\partial y^{\gamma}},
\end{eqnarray*}
where $(x^{i})$ are the coordinates on $M$ and $(x^{i},y^{\alpha})$ the coordinates on $E$.

Note that for the special cases $(E=TQ,[\cdot,\cdot],\rho=Id)$ and $(\mathfrak{g},[\cdot,\cdot],\rho=0)$ we recover the Euler-Lagrange equations and the Euler-Poincar\'{e} equations respectively.

\subsection{The inverse problem} \label{inverse}
In this section we recover the Helmholtz conditions for a SODE on a Lie algebroid and give a characterization of the inverse problem for regular Lie algebroids.

Let $\Gamma$ be a SODE on $E$, locally written as $\Gamma=y^{\alpha}\tilde{T}_{\alpha}+\Gamma^{\alpha}\tilde{V}_{\alpha}$. The inverse problem poses the following question: When is it possible to find a nondegenerate matrix of multipliers $g_{\alpha\beta}(x,y)$ such that
\begin{equation} \label{problem}
g_{\alpha\beta}(\dot{y}^{\alpha}-\Gamma^{\alpha})=\frac{d}{dt}\left( \frac{\partial L}{\partial y^{\beta}} \right)-\rho^{i}_{\beta}\frac{\partial L}{\partial x^{i}}+C^{\gamma}_{\beta\nu}y^{\nu}\frac{\partial L}{\partial y^{\gamma}}
\end{equation}
has a regular solution $L$? If it is possible then $\Gamma$ is called variational.

Given a SODE $\Gamma$ on $E$ and a local diffeomorphism $F:E\longrightarrow E^{*}$, we define a section of $(\mathcal{L}^{\tau}E)^{*}\longrightarrow E$ by $\Theta_{\Gamma,F}:=A_{E}\circ\mathcal{L}F\circ\Gamma$:
$$
\xymatrix{
\mathcal{L}^{\tau}E \ar[r]^{\mathcal{L}F} \ar[r]^{\mathcal{L}F} & \mathcal{L}^{\tau^{*}}E \ar[r]^{A_{E}} & (\mathcal{L}^{\tau}E)^{*} \\
E \ar[r]^{F} \ar[u]^{\Gamma} \ar[urr]_{\Theta_{\Gamma,F}} & E^{*}
}
$$

In local coordinates the diagram is the following:
$$
\xymatrix{
(x^{i},y^{\alpha},y^{\alpha},\Gamma^{\alpha}) \ar[r]^/-15pt/{\mathcal{L}F} & (x^{i},F_{\alpha},y^{\alpha},\frac{\partial F_{\alpha}}{\partial x^{i}}\rho^{i}_{\beta}y^{\beta}+\frac{\partial F_{\alpha}}{\partial y^{\beta}}\Gamma^{\beta}) \ar[r]^/-15pt/{A_{E}} & (x^{i},y^{\alpha},\frac{\partial F_{\alpha}}{\partial x^{i}}\rho^{i}_{\beta}y^{\beta}+\frac{\partial F_{\alpha}}{\partial y^{\beta}}\Gamma^{\beta}+C^{\gamma}_{\alpha\beta}F_{\gamma}y^{\beta},F_{\alpha}) \\
(x^{i},y^{\alpha}) \ar[r]^/+5pt/{F} \ar[u]^{\Gamma} \ar[urr]_{\Theta_{\Gamma,F}} & (x^{i},F_{\alpha}) 
}
$$

Let $\left\{\tilde{T}^{\gamma},\tilde{V}^{\gamma}\right\}$ denote the dual basis of $\{\tilde{T}_{\gamma},\tilde{V}_{\gamma}\}$. Then locally we can write $\Theta_{\Gamma,F}=\theta_{\alpha}\tilde{T}^{\alpha}+F_{\alpha}\tilde{V}^{\alpha}$, where 
\begin{equation}\label{eq:thetaAlpha}
 \theta_{\alpha}=\frac{\partial F_{\alpha}}{\partial x^{i}}\rho^{i}_{\beta}y^{\beta}+\frac{\partial F_{\alpha}}{\partial y^{\beta}}\Gamma^{\beta}+C^{\gamma}_{\alpha\beta}F_{\gamma}y^{\beta}.
\end{equation}

The differential of $\Theta_{\Gamma,F}$ is
$$
d^{\mathcal{L}^{\tau}E}\Theta_{\Gamma,F}=\left( \frac{\partial \theta_{\gamma}}{\partial x^{i}}\rho^{i}_{\beta}-\frac{1}{2}\theta_{\alpha}C^{\alpha}_{\beta\gamma} \right)\tilde{T}^{\beta}\wedge \tilde{T}^{\gamma}+\frac{\partial\theta_{\gamma}}{\partial y^{\beta}}\tilde{V}^{\beta}\wedge \tilde{T}^{\gamma}+\frac{\partial F_{\gamma}}{\partial x^{i}}\rho^{i}_{\beta}\tilde{T}^{\beta}\wedge \tilde{V}^{\gamma}+\frac{\partial F_{\beta}}{\partial y^{\gamma}}\tilde{V}^{\beta}\wedge \tilde{V}^{\gamma}.
$$ 
Imposing $d^{\mathcal{L}^{\tau}E}\Theta_{\Gamma,F}=0$ we obtain the Helmholtz conditions
\begin{equation} \label{Helmholtz}
\frac{\partial F_{\beta}}{\partial y^{\gamma}}=\frac{\partial F_{\gamma}}{\partial y^{\beta}}, \quad \frac{\partial\theta_{\gamma}}{\partial y^{\beta}}=\frac{\partial F_{\beta}}{\partial x^{i}}\rho^{i}_{\gamma}, \quad \frac{\partial \theta_{\gamma}}{\partial x^{i}}\rho^{i}_{\beta}-\frac{1}{2}\theta_{\alpha}C^{\alpha}_{\beta\gamma}=\frac{\partial \theta_{\beta}}{\partial x^{i}}\rho^{i}_{\gamma}-\frac{1}{2}\theta_{\alpha}C^{\alpha}_{\gamma\beta}.
\end{equation}

As mentioned earlier, these conditions are not enough to guarantee the existence of a Lagrangian function on $E$, since the Poincar\'{e} lemma does not hold for an arbitrary Lie algebroid. We need to ask for the additional condition
$$
\Theta_{\Gamma,F}(Z)=0 \mbox{ for all } Z\in \Gamma(\mbox{Ker}(\rho^{\tau})).
$$

Let $\left\{ e_{I} \right\}$ denote a local basis of $\Gamma(\mbox{Ker}(\rho))$. Then $\left\{ \tilde{T}_{I} \right\}$ is a local basis of $\Gamma(\mbox{Ker}(\rho^{\tau}))$ and the condition on $\Gamma(\mbox{Ker}(\rho^{\tau}))$ is
\begin{equation}\label{eq:Ker}
\theta_I=\frac{\partial F_{I}}{\partial x^{i}}\rho^{i}_{\beta}y^{\beta}+\frac{\partial F_{I}}{\partial y^{\beta}}\Gamma^{\beta}+C^{\gamma}_{I\beta}F_{\gamma}y^{\beta}=0, \quad I=1,\ldots,d=\mbox{dim}(\mbox{Ker}(\rho)) \leq n.
\end{equation}
Using the local basis $\{e_I,e_a\}$ adapted to ${\rm Ker} (\rho)$, the anchor map has the local expression $\rho^i_I=0$. Then Helmholtz conditions in~\eqref{Helmholtz} become
\begin{align} 
\frac{\partial F_{\beta}}{\partial y^{\gamma}}&=\frac{\partial F_{\gamma}}{\partial y^{\beta}}, \nonumber \\
\frac{\partial\theta_{a}}{\partial y^{\beta}}&=\frac{\partial F_{\beta}}{\partial x^{i}}\rho^{i}_{a}, \quad 
\frac{\partial\theta_{I}}{\partial y^{\beta}}=0, \label{eq:HKer1} \\
\frac{\partial \theta_{a}}{\partial x^{i}}\rho^{i}_{b}&-\theta_{\alpha}C^{\alpha}_{ba}=\frac{\partial \theta_{b}}{\partial x^{i}}\rho^{i}_{a}, \quad 
\frac{\partial \theta_{I}}{\partial x^{i}}\rho^{i}_{a}-\theta_{\alpha}C^{\alpha}_{aI}=0, \quad  \theta_{\alpha}C^{\alpha}_{JI}=0. \label{eq:HKer2}
\end{align}
From the second equation in~\eqref{eq:HKer1} we deduce that $\theta_I(x,y)=\theta_I(x)$. Then the additional condition
 in~\eqref{eq:Ker} will
be satisfied if $\theta_I(x)=0$. 

\begin{thm}
A SODE section $\Gamma$ on a regular Lie algebroid $E$ is variational if and only if there is a local diffeomorphism $F: E\longrightarrow E^{*}$ such that $d^{\mathcal{L}^{\tau}E}\Theta_{\Gamma,F}=0$ and $\Theta_{\Gamma,F}(Z)=0$ for all $Z\in\Gamma(Ker(\rho^{\tau}))$.
\end{thm}

\begin{proof}
$\Leftarrow$ If there is a local diffeomorphism $F$ such that $d^{\mathcal{L}^{\tau}E}\Theta_{\Gamma,F}=0$ and $\Theta_{\Gamma,F}(Z)=0$ for all $Z\in\Gamma(Ker(\rho^{\tau}))$ then by Lemma \ref{lemma-sec} we have $\Theta_{\Gamma,F}=d^{\mathcal{L}^{\tau}E}L$ for a locally defined function $L:E\longrightarrow \mathbb{R}$. In local coordinates we get 
$$
F_{\beta}=\frac{\partial L}{\partial y^{\beta}} \quad \mbox{ and } \quad \frac{\partial F_{\gamma}}{\partial y^{\beta}}\Gamma^{\beta}+y^{\beta}\frac{\partial F_{\gamma}}{\partial x^{i}}\rho^{i}_{\beta}+y^{\beta}F_{\alpha}C^{\alpha}_{\gamma\beta}=\frac{\partial L}{\partial x^{i}}\rho^{i}_{\gamma}.
$$
Therefore $\frac{d}{dt}\left( \frac{\partial L}{\partial y^{\beta}} \right)-\rho^{i}_{\beta}\frac{\partial L}{\partial x^{i}}+C^{\gamma}_{\beta \nu}y^{\nu}\frac{\partial L}{\partial y^{\gamma}}=\frac{\partial F_{\beta}}{\partial y^{\gamma}}(\dot{y}^{\gamma}-\Gamma^{\gamma})$, $g_{\beta\gamma}=\frac{\partial F_{\beta}}{\partial y^{\gamma}}$ are the multipliers for the problem and $L$ is regular since $(g_{\beta\gamma})$ is non-degenerate.

$\Rightarrow$ If $\Gamma$ is variational then there is a regular Lagrangian $L$ such that equation (\ref{problem}) 
is satisfied with $g_{\alpha\beta}=\frac{\partial^{2}L}{\partial y^{\alpha}\partial y^{\beta}}$. Taking $F$ to be the
 Legendre transformation, which is a local diffeomorphism, it is straightforward to check that equations (\ref{Helmholtz}) are
 satisfied using $\theta_{\gamma}=\frac{\partial L}{\partial x^{i}}\rho^{i}_{\gamma}$ and the structure equation 
$\rho^{j}_{\alpha}\frac{\partial\rho^{i}_{\beta}}{\partial x^{j}}-\rho^{j}_{\beta}\frac{\partial \rho^{i}_{\alpha}}{\partial x^{j}}=\rho^{i}_{\gamma}C^{\gamma}_{\alpha\beta}$ for the last set. If $Z\in\Gamma(Ker(\rho^{\tau}))$, $Z=z^{\alpha}\tilde{T}_{\alpha}$, then we also get
$$
\Theta_{\Gamma,F}(Z)=\left(\frac{\partial^{2}L}{\partial y^{\beta} \partial y^{\gamma}}\Gamma^{\beta}+y^{\beta}\frac{\partial^{2}L}{\partial x^{i}\partial y^{\gamma}}\rho^{i}_{\beta}+y^{\beta}\frac{\partial L}{\partial y^{\alpha}}C^{\alpha}_{\gamma\beta}\right)z^{\gamma}=\frac{\partial L}{\partial x^{i}}\rho^{i}_{\gamma}z^{\gamma}=0.
$$
\end{proof}

\begin{ex} \label{ex:tangent}
Note that for the Lie algebroid $(E=TQ,[\cdot,\cdot],\rho=Id)$, where $[\cdot,\cdot]$ is the Lie bracket of vector fields, we recover the equations
$$
\frac{\partial F_{\beta}}{\partial y^{\gamma}}=\frac{\partial F_{\gamma}}{\partial y^{\beta}}, \quad \frac{\partial \Gamma(F_{\gamma})}{\partial y^{\beta}}=\frac{\partial F_{\beta}}{\partial x^{\gamma}}, \quad \frac{\partial\Gamma(F_{\gamma})}{\partial x^{\beta}}=\frac{\partial\Gamma(F_{\beta})}{\partial x^{\gamma}}
$$
given in \cite{BFM} and the condition involving $\mbox{Ker}(\rho^{\tau})$ is void. 
\end{ex}

\begin{ex} \label{ex:algebra}
For a Lie algebra $(\mathfrak{g},[\cdot,\cdot],\rho=0)$ we get the Helmholtz conditions \cite{CM2008}
$$
\frac{\partial F_{\beta}}{\partial y^{\gamma}}=\frac{\partial F_{\gamma}}{\partial y^{\beta}}, \quad \frac{\partial\left( \frac{\partial F_{\gamma}}{\partial y^{\tau}}\Gamma^{\tau}+C^{r}_{\gamma \tau}F_{r}y^{\tau} \right)}{\partial y^{\beta}}=0, \quad \left( \frac{\partial F_{\alpha}}{\partial y^{\tau}}\Gamma^{\tau}+C^{r}_{\alpha \tau}F_{r}y^{\tau} \right) C^{\alpha}_{\beta \gamma}=0.
$$
In this case the condition on $\Gamma(\mbox{Ker}(\rho^{\tau}))$ is $\frac{\partial F_{\alpha}}{\partial y^{\tau}}\Gamma^{\tau}+C^{r}_{\alpha \tau}F_{r}y^{\tau}=0$, which makes the last two conditions always true. Note that the symmetry gives a function $L$ such that $F_{\gamma}=\frac{\partial L}{\partial y^{\gamma}}$ and then the remaining conditions are just the Euler-Poincar\'{e} equations for $L$.
\end{ex}

We will use the following term in order to avoid confusion: 
\begin{definition}
A SODE $\Gamma$ on $E$ will be called weak variational if there is a local diffeomorphism $F:E\longrightarrow E^{*}$ such that $d^{\mathcal{L}^{\tau}E}\Theta_{\Gamma,F}=0$.
\end{definition}

Hence, a SODE $\Gamma$ on $E$ is variational if it is weak variational and $\Theta_{\Gamma,F}(Z)=0$ for all $Z\in\Gamma(\mbox{Ker}(\rho^{\tau}))$. This definition, for the case of a Lie algebra, is equivalent to satisfying the reduced Helmholtz conditions given in \cite{CM2008}.

Due to the lack of a Poincar\'{e} lemma we give a generalization of the Theorem by Crampin in \cite{81Crampin} for weak variational SODEs substituting the closedness condition by local exactness of a section of the bundle $(\mathcal{L}^{\tau}E)^{*}\wedge (\mathcal{L}^{\tau}E)^{*}\longrightarrow E$, which plays the role of the Cartan 2-section generalizing the Poincar\'{e}-Cartan 2-form.
 
\begin{thm} \label{crampin-alg}
A SODE $\Gamma$ on a regular Lie algebroid $E$ is weak variational if and only if there is a nondegenerate section $\Omega$ of $(\mathcal{L}^{\tau}E)^{*}\wedge (\mathcal{L}^{\tau}E)^{*}\longrightarrow E$ such that 
\begin{itemize}
\item $\mathcal{L}_{\Gamma}\Omega=0$,
\item $\Omega=d^{\mathcal{L}^{\tau}E}\Theta$ for some locally defined section $\Theta$ of $(\mathcal{L}^{\tau}E)^{*}\longrightarrow E$ ,
\item $\Omega(\tilde{V}_{\alpha},\tilde{V}_{\beta})=0$ for all $\alpha, \beta$.
\end{itemize}
\end{thm}

\begin{proof}
$\Rightarrow$ If $\Gamma$ is weak variational then there is a local diffeomorphism $F:E\longrightarrow E^{*}$ over $M$ such that $d^{\mathcal{L}^{\tau}E}\Theta_{\Gamma,F}=0$. Then define $\Omega=d^{\mathcal{L}^{\tau}E}(F^{*}\lambda_{E})$ which clearly satisfies the second condition and $\Omega(\tilde{V}_{\alpha},\tilde{V}_{\beta})=0$. Note that $\mathcal{L}_{\Gamma}F^{*}\lambda_{E}=\Theta_{\Gamma,F}$ and hence it also satisfies $\mathcal{L}_{\Gamma}\Omega=d^{\mathcal{L}^{\tau}E}\Theta_{\Gamma,F}$=0. Finally the non-degeneracy of $\left( \frac{\partial F_{\gamma}}{\partial y^{\beta}} \right)$ implies the non-degeneracy of $\Omega$.

$\Leftarrow$ If we write $\Theta=\mu_\alpha\tilde{T}^{\alpha}+\nu_\alpha \tilde{V}^{\alpha}$ then the condition $d^{\mathcal{L}^{\tau}E}\Theta(\tilde{V}_{\alpha},\tilde{V}_{\beta})=\frac{\partial \nu_{\alpha}}{\partial y^{\beta}}-\frac{\partial \nu_\beta}{\partial y^\alpha}=0$ gives a locally defined function $f:E\longrightarrow \mathbb{R}$ such that $\nu_\alpha=\frac{\partial f}{\partial y^\alpha}$ and then $d^{\mathcal{L}^{\tau}E}f(\tilde{V}_{\alpha})=\Theta(\tilde{V}_{\alpha})=\nu_\alpha$. Define $\tilde{\Theta}=\Theta-d^{\mathcal{L}^{\tau}E}f$, which satisfies $\tilde{\Theta}(\tilde{V}_{\alpha})=0$ and $d^{\mathcal{L}^{\tau}E}\tilde{\Theta}=\Omega$. We seek to have $\tilde{\Theta}=F^{*}\lambda_{E}$ for some local diffeomorphism $F$, so we define $F:E\longrightarrow E^{*}$ by $\langle F(v_x), w_x\rangle=\langle \tilde{\Theta}(v_x),W_x\rangle$, where $x\in M$, $v_x,w_x\in E$ and $W_x\in\mathcal{L}^{\tau}E$ is such that $\tau^{\tau}(W_x)=w_x$. This definition does not depend on the choice of $W_x$ since $\tilde{\Theta}$ vanishes on vertical sections. Locally if we write $\tilde{\Theta}=A_{\alpha}\tilde{T}^{\alpha}+B_{\alpha}\tilde{V}^{\alpha}$, then $\tilde{\Theta}=\left( \frac{\partial f}{\partial x^{i}}\rho^{i}_{\alpha}-A_{\alpha} \right)\tilde{T}^{\alpha}=:F_{\alpha}\tilde{T}^{\alpha}$ and the non-degeneracy of $\left( \frac{\partial F_{\alpha}}{\partial y^{\beta}} \right)$ follows from the non-degeneracy of $\Omega$. Finally $d^{\mathcal{L}^{\tau}E}\Theta_{\Gamma,F}=d^{\mathcal{L}^{\tau}E}\mathcal{L}_{\Gamma}F^{*}\lambda_{E}=d^{\mathcal{L}^{\tau}E}\mathcal{L}_{\Gamma}\tilde{\Theta}=\mathcal{L}_{\Gamma}\Omega=0$, that is, $\Gamma$ is weak variational.
\end{proof}

In \cite{CM2008} some variational examples are found by requiring only that the Helmholtz conditions are  satisfied, but this is not generally the case. As we have seen, in order to guarantee the existence of a Lagrangian for a SODE on a Lie algebroid we need to ask for an extra condition. Next we give an example of a SODE on a Lie algebra which is weak variational but not variational.

\begin{ex}\label{examples}
Let $(y^{1},y^{2},y^{3})$ denote the coordinates for $\mathfrak{g}=\mathfrak{se}(2)$ corresponding to the basis given in Example \ref{se2} and define the following SODE:

$$
\begin{array}{cccc}
\Gamma: & \mathfrak{se}(2) & \longrightarrow & \mathcal{L}^{\tau}\mathfrak{se}(2)\cong 2\mathfrak{g} \\& (y^{1},y^{2},y^{3}) & \longmapsto & (y^{1},y^{2},y^{3},\Gamma^{1}=y^{2}y^{3},\Gamma^{2}=-y^{1}y^{3},\Gamma^{3}=1)
\end{array}
$$

Consider the local diffeomorphism from $\mathfrak{se}(2)$ to $\mathfrak{se}(2)^{*}$ given by $F_{1}=y^{1}, F_{2}=y^{2}, F_{3}=y^{3}$ and compute $\theta_{1}=\theta_{2}=0$ and $\theta_{3}=1$ to get $\Theta_{\Gamma,F}=\tilde{T}^{3}+F_{\alpha}\tilde{V}^{\alpha}$. Then $d^{\mathcal{L}^{\tau}\mathfrak{se}(2)}\Theta_{\Gamma,F}=-\frac{1}{2}\theta_{3}C^{3}_{\beta\gamma}\tilde{T}^{\beta}\wedge\tilde{T}^{\gamma}=0$ since $C^{3}_{\beta\gamma}=0$, that is, the Helmholtz conditions are satisfied, so $\Gamma$ is weak variational, BUT since $\Theta_{\Gamma,F}(e_{3}^{c})=1\not=0$, $\Gamma$ is not variational. 

The corresponding left-invariant SODE on $TG$ is given by
\begin{equation*}
\ddot{x}=0, \quad \ddot{y}=0, \quad \ddot{\theta}=1,
\end{equation*}
where $(x,y,\theta)$ are coordinates on $SE(2)$. According to \cite[Theorem 3]{CM2008} this SODE is variational, that is, we can find a Lagrangian on $TG$, but not an invariant one. It is actually straightforward to obtain the Lagrangian $L=\frac{1}{2}\left(\dot{x}^{2}+\dot{y}^{2}+\dot{\theta}^{2}\right)+\theta$.
\end{ex}

\begin{remark}
It is also possible to give an example of a SODE on a Lie algebra $\mathfrak{g}$ which is not variational on $\mathfrak{g}$ and also not weak variational but variational on $TG$. See Example 8.3 Case 2C in \cite{CM2008}.
\end{remark}

\section{Morphisms and the variational problem}\label{Sec:Morph}

The geometric description of the inverse problem on Lie algebroids given in the previous section leads to a generalization of some results in~\cite{CM2008}, where the relationship between the inverse problem on the tangent bundle to a Lie group and the corresponding reduced inverse problem on the Lie algebra is studied. By means of morphisms of Lie algebroids the same relationship can be studied for the inverse problem on Lie algebroids. Moreover, the proof of the following result is intrinsic, in contrast to the proof for the Lie group case in~\cite{CM2008}.

\begin{thm}\label{red}
Let $\Psi: E\rightarrow E'$ be 
 a morphism of Lie algebroids, and consider its prolongation 
${\mathcal L} \Psi: {\mathcal L}^\tau E\rightarrow {\mathcal L}^{\tau'} E'$. Let $\Gamma$ and $\Gamma'$ be  SODE sections on $E$ and $E'$ respectively such that
\[
{\mathcal L} \Psi\circ \Gamma=\Gamma'\circ \Psi\, .
\]
If $\Gamma'$ is weak variational (variational) then $\Gamma$ is weak variational (variational). 
\end{thm}

\begin{proof}
Since ${\mathcal L}\Psi$ is a morphism of Lie algebroids we have that 
$({\mathcal L}\Psi)^* d^{{\mathcal L}^{\tau'} E'}=d^{{\mathcal L}^\tau E}({\mathcal L}\Psi)^*$.
From Theorem \ref{crampin-alg} there exists an exact section $\Theta'\in\Gamma((\mathcal{L}^{\tau'}E')^{*})$ such that $\Omega'=d^{{\mathcal L}^{\tau'} E'}\Theta'$ satisfies
${\mathcal L}_{\Gamma'}\Omega'=0$ and the restriction of $\Omega'$ to vertical sections vanishes, that is, $\Omega' (U', V')=0$, where $U', V'$ are vertical sections ($S(U')=S(V')=0$). 

As $\Psi$ is a morphism of Lie algebroids, we have that 
$\Theta=({\mathcal L} \Psi)^*\Theta'$ also satisfies the conditions of Theorem \ref{crampin-alg}.
In fact, for every $Z\in {\mathcal L}^\tau E$
\begin{eqnarray*}
\langle {\mathcal L}_{\Gamma}\Theta, Z\rangle &=&
\rho^{\tau}(\Gamma)(\langle  ({\mathcal L}\Psi)^*\Theta', Z\rangle) -\langle \Theta',  {\mathcal L}\Psi([\Gamma, Z]^\tau)\rangle\\
&=&\rho^{{\tau}'}(\Gamma')(\langle \Theta', {\mathcal L}\Psi (Z)\rangle) -\langle \Theta',  [\Gamma', {\mathcal L}\Psi(Z)]^{\tau'}\rangle\\
&=&
\langle({\mathcal L} \Psi)^*({\mathcal L}_{\Gamma'}\Theta'), Z\rangle\, .
\end{eqnarray*}
Therefore,  ${\mathcal L}_{\Gamma}\Theta=({\mathcal L} \Psi)^*({\mathcal L}_{\Gamma'}\Theta')$ and also 
\[
{\mathcal L}_{\Gamma}\Omega=({\mathcal L} \Psi)^*({\mathcal L}_{\Gamma'}\Omega')=0\, .
\]
Moreover, for all $Z_1, Z_2\in {\mathcal L}^{\tau}E$ 
\begin{eqnarray*}
\Omega (S(Z_1), S(Z_2))&=&({\mathcal L}\Psi)^*\Omega' (S(Z_1), S(Z_2))
\\
&=&\Omega' ({\mathcal L}\Psi( S (Z_1)), {\mathcal L}\Psi( S (Z_2)))\\
&=&\Omega'(S'({\mathcal L}\Psi(Z_1)), S'( {\mathcal L}\Psi(Z_2)))\\
&=&0
\end{eqnarray*}
 using that 
${\mathcal L}\Psi\circ S=S'\circ {\mathcal L}\Psi$ (see~\cite{CoLeMaMa}).
This proves that if $\Gamma'$ is weak variational then  $\Gamma$ is also weak variational.
Obviously if $\Theta'$ is exact,  then $\Theta$ is also exact. Therefore, if $\Gamma'$ is variational then  $\Gamma$ is also variational.
\end{proof}

Now we write the converse to the previous result for the case of a fiberwise surjective morphism satisfying an extra assumption.

\begin{thm}
Let $\Psi: E\rightarrow E'$ 
be a fiberwise surjective morphism of Lie algebroids. Let $\Gamma$ and $\Gamma'$ be SODE sections on $E$ and $E'$ respectively such that
$
{\mathcal L} \Psi\circ \Gamma=\Gamma'\circ \Psi \, .
$
If $\Gamma$ is weak variational (variational) 
and it admits a solution  $\Theta$ of Theorem \ref{crampin-alg} such that $\Theta=({\mathcal L} \Psi)^*\Theta'$ for some $\Theta'\in\Gamma((\mathcal{L}^{\tau'}E')^{*})$, 
then $\Gamma'$ is weak variational (variational). 
\end{thm}

\begin{proof}
The proof follows the same lines that Theorem \ref{red} using that $\Psi$ is a fiberwise surjective morphism.
\end{proof}

\begin{remark}
See Example \ref{examples} for a case in which there is no section $\Theta'$ satisfying the property $\Theta=({\mathcal L} \Psi)^*\Theta'$.
\end{remark}

\section{The inverse problem for Atiyah algebroids}\label{Sec:Atiyah}

The theory developed in section~\ref{S:Lie} has a very interesting application when Atiyah algebroids are considered. We first review the main notions of Atiyah algebroids, see~\cite{LMM} and references therein for more details, and then we geometrically characterize the inverse problem on Atiyah algebroids.  As shown in~\cite{LMM} the Euler-Lagrange equations of a $G$-invariant Lagrangian can be reduced to Lagrange-Poincar\'e equations by using the morphism of Lie algebroids between $TQ$ and $TQ/G$ (see \cite{cendra}).
Thus the results in section~\ref{Sec:Morph} can be applied to establish some relationship between the inverse problem and its reduced version.

\subsection{Atiyah algebroid  associated to a principal  bundle}

Let $\pi: Q\rightarrow Q/G=M$ be a principal $G$-bundle and $\Phi: G\times Q\rightarrow Q$, $\Phi_g(q)=\Phi(g,q)$, the 
corresponding $G$-action. Denote by $\Phi^T: G\times TQ\rightarrow TQ$ the tangent lift of $\Phi$, that is, $\Phi^T_g=T\Phi_g$ for all $g\in G$. Now 
consider the quotient vector bundle
$\tau_{Q/ G}: TQ/G\rightarrow M$ whose space of sections $\Gamma(TQ/G)$ is identified with the $G$-invariant vector fields on $Q$. 

Let ${\mathfrak g}$ be the Lie algebra of $G$ and take the action of $G$ on $Q\times {\mathfrak g}$ given by
\[
\begin{array}{rcl}
G\times (Q\times {\mathfrak g}) &\longrightarrow& Q\times {\mathfrak g}\\
(g, (q, \xi))&\longmapsto& (\Phi_g(q), \hbox{Ad}_g(\xi)),\,
\end{array}
\]
where ${\rm Ad}\colon G\times {\mathfrak g} \rightarrow {\mathfrak g}$ is the adjoint representation of $G$ on ${\mathfrak g}$.
The quotient vector bundle $\tilde{\mathfrak g}=(Q\times {\mathfrak g})/G$ is called the adjoint bundle associated with 
the principal bundle $\pi\colon Q\rightarrow M$. 
If $\xi_Q$ is the infinitesimal generator of the action $\Phi$ associated with $\xi \in {\mathfrak g}$, that is, 
\[
\xi_Q(q)=\frac{d}{dt}\Big|_{t=0}\Phi(\hbox{exp}(t\xi), q),
\]
then we have the following monomorphism of vector bundles: 
\[
\begin{array}{rrcl}
j:&\tilde{\mathfrak g}&\longrightarrow& TQ/G\\
&[(q, \xi)]&\longmapsto&[\xi_Q(q)].
\end{array}
\]
Moreover, we have the following exact sequence called the Atiyah sequence~\cite{MAC}:
\[
0\longrightarrow \tilde{\mathfrak g}\xrightarrow{ \ j\  } TQ/G
\xrightarrow{\ [T\pi]\ }TM\longrightarrow 0 .
\]
Assume that we have a principal connection $A$ on $Q$, that is, $A: TQ\rightarrow {\mathfrak g}$ satisfying $A(\xi_Q(q))=\xi$ and  $A$ is equivariant with respect to the actions $\phi^{T}: G
\times TQ \to TQ$ and $Ad: G \times {\mathfrak g} \to {\mathfrak g}$. Every principal connection $A$ induces the following vector bundle
isomorphism over the identity:
\[
\begin{array}{rcl}
TQ/G&\longrightarrow & T(Q/G)\oplus {\tilde{\mathfrak g}}\\
\;[X_q]&\longmapsto& T_q\pi(X_q)\oplus [(q, A(X_q))] ,
\end{array}
\]
where $X_q\in T_qQ$. 
Therefore we have an identification $\Gamma(TQ/G)\cong {\mathfrak X}(M)\oplus \Gamma (\tilde{\mathfrak g})$, where $\Gamma (\tilde{\mathfrak g})$ is identified with the set of vector fields on $Q$ which are $\pi$-vertical and $G$-invariant. 
Let $B:TQ\oplus TQ\to {\mathfrak g}$ be the 
curvature of the connection $A$ in the principal bundle $\pi$. The Lie bracket 
$\lcf\cdot,\cdot\rcf$ on $\Gamma(TQ/G)\cong \Gamma(TM\oplus
\tilde{\mathfrak g})\cong {\mathfrak X}(M)\oplus \Gamma (\tilde{\mathfrak g})$ is defined as follows
\[
\lcf X\oplus \tilde{\xi},Y\oplus \tilde{\eta}\rcf=[X,Y] \oplus
([\tilde{\xi},\tilde{\eta}] +
[X^h,\tilde{\eta}]-[Y^h,\tilde{\xi}]-B(X^h,Y^h)),
\]
for $X,Y\in {\mathfrak X}(M)$ and $\tilde{\xi}, \tilde{\eta}\in \Gamma (\tilde{\mathfrak g}),$ where $X^h, Y^h\in {\mathfrak X}(Q)$ are the horizontal lift of $X$, $Y$, respectively,
via the principal connection $A$. The anchor map
$\rho:\Gamma (TQ/G)\cong {\mathfrak X}(M)\oplus\Gamma (\tilde{\mathfrak g})\to
{\mathfrak X}(M)$ is given by
\[
\rho(X\oplus \tilde{\xi})=X.
\]

Now we will give a local description. Let $U\times G$ be a local trivialization of the principal bundle $\pi:Q\to M$
 where $U$  is an open subset of $M$ with local coordinates $(x^i)$. Then we
consider the trivial principal bundle $\pi:U\times G\to U$, where the action of $G$ on $U\times G$ is given by left
multiplication on the second factor, that is, $\Phi_g(m, h)=(m, gh)$, where $m\in U$ and $g, h\in G$. For a basis 
$\{\xi_a\}$  of ${\mathfrak g}$, $1\leq a \leq n$,  we denote by $\{\lvec{\xi_a}\}$ the
corresponding left-invariant vector fields on $G$.
Then the principal connection is specified by coefficients $A^a_i(x)$ satisfying
\[
A\left(\frac{\partial }{\partial x^i}\Big|_{{(x,e)}}\right)=A_i^a(x)\xi_a,   \quad  1\leq i \leq m,
\]
where $x\in U$ and $e$ is the identity element of $G$. The horizontal lift
of a coordinate vector field  $\displaystyle\frac{\partial }{\partial x^i}$
on $U$ is the vector field $\left(\displaystyle\frac{\partial
}{\partial x^i}\right)^h$ on $U\times G$ given by
$
\displaystyle{\left(\frac{\partial }{\partial x^i}\right)^h=\frac{\partial }{\partial
x^i}-A_i^a(x)\lvec{\xi_a}}.
$
Thus, the vector fields on $U\times G$
\begin{equation}\label{basisUG}
\left\{e_i=\frac{\partial }{\partial x^i}-A_i^a\lvec{\xi_a},e_b=\lvec{\xi_b}\right\}
\end{equation}
are left $G$-invariant and define a local basis $\{e'_{i},
e'_{b}\}$ of $\Gamma(TQ/G)\cong {\mathfrak X}(M)\oplus
\Gamma(\tilde{\mathfrak g}).$ We will denote by $(x^i, y^i,y^b)$ the
corresponding fibered coordinates on $TQ/G$.

The curvature of the principal connection is given by 
\[
B\left(\displaystyle\frac{\partial
}{\partial x^i}\Big|_{(x,e)}, \displaystyle\frac{\partial }{\partial
x^j}\Big|_{(x,e)}\right)=B_{ij}^a(x)\xi_a,
\]
for $i,j\in \{1,\dots , m\}$ and $x\in U$. If $c_{ab}^c$ are the
structure constants of ${\mathfrak g}$ with respect to the basis
$\{\xi_a\}$, then
\[B_{ij}^c = \frac{\partial A_i^c}{\partial x^j}-\frac{\partial
A_j^c}{\partial x^i}-c_{ab}^cA_i^aA_j^b.
\]
Then for the previous local basis  $\{e'_i,e'_b\}$ of $\Gamma(TQ/G)$
we deduce that
\[
\lcf e'_i,e'_j\rcf=-B_{ij}^ce'_c,\;\;\; \lcf
e'_i,e'_a\rcf=c_{ab}^cA_{i}^be'_c, \; \; \lcf
e'_a,e'_b\rcf=c_{ab}^ce'_c,\;\;\;\]
 \[ \rho(e'_i)=\frac{\partial
}{\partial x^i}, \;\;\; \rho(e'_a)=0,\] for $i,j\in \{1,\dots
,m\}$ and $a,b\in \{1,\dots ,n\}.$ 
Thus, the local structure
functions of the Atiyah algebroid $\tau_{{Q}/G}: TQ/G\to M=Q/G$
with respect to the local coordinates $(x^i)$ and to the local
basis $\{e'_i,e'_a\}$ of $\Gamma(TQ/G)$ are
\begin{equation}\label{Aa}
\begin{array}{c}
C_{ij}^k= C_{ia}^j=-C_{ai}^j=
C_{ab}^i=0,\;\;\;C_{ij}^a=-B_{ij}^a,\;\;\;
C_{ia}^c=-C_{ai}^c=c_{ab}^cA_i^b,\;\;\;
C_{ab}^c=c_{ab}^c,\\
\rho_i^j=\delta_{ij},\;\;\; \rho_i^a=\rho_a^i=\rho_a^b=0.
\end{array}
\end{equation}

\subsection{The inverse problem for Atiyah algebroids}

In the case of an Atiyah algebroid the section $\Theta_{\Gamma, F}=\theta_{\alpha}\tilde{T}^{\alpha}+F_{\alpha} \tilde{V}^{\alpha}=\theta_i \tilde{T}^i+\theta_a \tilde{T}^a+F_i\tilde{V}^i+F_a\tilde{V}^a
 $ of $(\mathcal{L}^{\tau}E)^{*}\rightarrow E$ defined at the beginning of section \ref{inverse} has the following local components:
 \begin{eqnarray*}
 \theta_i&=&\frac{\partial F_
 i}{\partial x^j}y^j+\frac{\partial F_i}{\partial y^j}\Gamma^j+\frac{\partial F_i}{\partial y^a}\Gamma^a-B_{ij}^aF_ay^j
 +c_{ab}^cA^b_iF_cy^a,\\
 \theta_a&=&\frac{\partial F_a}{\partial x^j}y^j+\frac{\partial F_a}{\partial y^j}\Gamma^j+\frac{\partial F_a}{\partial y^b}\Gamma^b
 -c_{ab}^cA^b_iF_cy^i+c_{ab}^cF_cy^b.\\
 \end{eqnarray*}
In this case the Helmholtz conditions, given by $d^{{\mathcal L}^{\tau_{Q/G}}}\Theta_{\Gamma, F}=0$, are
\begin{align*}
\frac{\partial F_{\beta}}{\partial y^{\gamma}}&=\frac{\partial F_{\gamma}}{\partial y^{\beta}}\; , \\
\frac{\partial \theta_{j}}{\partial y^{\beta}}&=\frac{\partial F_{\beta}}{\partial y^{j}}\; , \quad\frac{\partial \theta_{b}}{\partial y^{\beta}}=0\; ,\\
\frac{\partial \theta_{i}}{\partial x^{j}}&
+\theta_{a}B^{a}_{ji}=
\frac{\partial \theta_{j}}{\partial x^{i}}\; , \quad  \frac{\partial \theta_{b}}{\partial x^{i}}=\theta_{a}c^{a}_{bd}A^d_i\; , \quad \theta_cc^c_{ab}=0\; , 
\end{align*}
compared with~\eqref{eq:HKer1} and~\eqref{eq:HKer2}. From the last two equations we conclude that $\partial \theta_b /\partial x^i=0$. Thus $\theta_b$ is a constant function. 

The extra condition for exactness of $ \Theta_{\Gamma, F}$ is given by
$$
\theta_a=\frac{\partial F_{a}}{\partial x^{i}}y^{i}+\frac{\partial F_{a}}{\partial y^{\beta}}\Gamma^{\beta}-c^d_{ab}A^b_iF_{d}y^{i}
+c^{d}_{ab}F_{d}y^{b}
=0\; ,
$$
since ${\rm Ker} \rho={\rm span} \{ e_a'\}$. Thus $\Theta_{\Gamma,F}$ is exact if $\theta_a(x,y)=\theta_a=0$. Recall that in the general Lie algebroid case from section~\ref{S:Lie} the condition of exactness was $\theta_a(x,y)=\theta_a(x)=0$.

\begin{remark}
Note that among the set of Helmholtz conditions (HC), the equations
$$
\frac{\partial \theta_{b}}{\partial y^{\beta}}=0, \quad \frac{\partial \theta_{b}}{\partial x^{i}}=\theta_{a}c^{a}_{bd}A^d_i \quad \mbox{and} \quad \theta_cc^c_{ab}=0
$$
are implied by the extra condition $\theta_a=0$. Then a set of necessary and sufficient conditions for variationality is
$$
\frac{\partial F_{\beta}}{\partial y^{\gamma}}=\frac{\partial F_{\gamma}}{\partial y^{\beta}}, \quad \frac{\partial \theta_{j}}{\partial y^{\beta}}=\frac{\partial F_{\beta}}{\partial y^{j}}, \quad \frac{\partial \theta_{i}}{\partial x^{j}}=\frac{\partial \theta_{j}}{\partial x^{i}}, \quad \theta_a=0,
$$
in contrast to Example \ref{ex:tangent} where HC are necessary and sufficient and Example \ref{ex:algebra} where most 
HC are implied by the extra condition on the kernel. In this example the kernel of $\rho$ is not trivial and neither the whole domain so we get less overlap between the two sets of conditions.
\end{remark}

\begin{remark}
If we are given a SODE on $TQ$ and a SODE on $TQ/G$ related as in Theorem \ref{red} then it is enough to solve the Helmholtz conditions on the Atiyah algebroid in order to get a Lagrangian on $TQ$ since on tangent bundles the notions of variational and weak variational coincide.
\end{remark}

\section{Conclusions and future developments}

The contributions of this paper include a characterization of the inverse problem of the calculus of variations on regular Lie algebroids using
Lagrangian submanifolds. One of the advantages of our approach is the easy adaptability to different cases. In particular, 
in future work we will study the following extensions: 
\begin{itemize}
\item The inverse problem for nonholonomic systems on Lie algebroids using isotropic submanifolds, similarly to the description 
on the tangent bundle given in~\cite{BFM}. 

\item We will carefully study the relationship between our techniques and hamiltonization of nonholonomic systems on Lie algebroids. 
This is useful to study invariance properties of the nonholonomic flow (preservation of a volume form, symmetries...).

\item Another interesting possibility is to extend our technique, always using Lagrangian and isotropic submanifolds,
now for Lie groupoids $G$~\cite{WEINST}. 
This case will be useful to study the inverse problem for discrete systems, that is, when a second-order difference equation 
can be derived as the  flow associated to the discrete Euler-Lagrange equations for a discrete Lagrangian 
$L_d: G\rightarrow \R$ (see \cite{2001MaWe}).
 \end{itemize}

\appendix

\section{Relation to other approaches} \label{approaches}

In section~\ref{inverse} we recover the Helmholtz conditions given in \cite{POP} as the vanishing of $d^{\mathcal{L}^{\tau}E}\Theta_{\Gamma,F}$ on a certain basis of sections of $\mathcal{L}^{\tau}E\longrightarrow E$.

In the previous section we worked in the basis $\left\{\tilde{T}_{\alpha},\tilde{V}_{\alpha}\right\}$ of local sections of $\mathcal{L}^{\tau}E$, constructed from a basis $\{e_{\alpha}\}$ of local sections of $E$. Another common basis of sections of $\mathcal{L}^{\tau}E$ is $\{e_{\alpha}^{C},e_{\alpha}^{V}\}$, the set of complete and vertical lifts of $\{e_{\alpha}\}$. The relationship between both is 
$$
\tilde{T}_{\alpha}=e_{\alpha}^{C}+C^{\gamma}_{\alpha\beta}y^{\beta}\tilde{V}_{\gamma} \quad \mbox{ and } \quad \tilde{V}_{\gamma}=e_{\gamma}^{V}.
$$

As in the tangent bundle case, a SODE on a Lie algebroid defines a connection (see \cite{POP}). Then the horizontal lift of a section $X\in\Gamma(E)$ can be defined from its complete and vertical lift and the SODE as
$$
X^{H}=\frac{1}{2}\left(X^{C}-[\Gamma,X^{V}]\right),
$$
and we get another basis $\{H_{\alpha}:=e_{\alpha}^{H},e_{\alpha}^{V}\}$ of sections of $\mathcal{L}^{\tau}E$. The relationship with the above is given by
$$
H_{\alpha}=e_{\alpha}^{H}=\tilde{T}_{\alpha}+\frac{1}{2}\left(\frac{\partial \Gamma^{\gamma}}{\partial y^{\alpha}}-C^{\gamma}_{\alpha\beta}y^{\beta} \right)\tilde{V}_{\gamma}=\tilde{T}_{\alpha}+\Lambda^{\gamma}_{\alpha}\tilde{V}_{\gamma}.
$$

Note that if $\Gamma$ is variational we have $\mathcal{L}_{\Gamma}F^{*}\lambda_{E}=\Theta_{\Gamma,F}$ for some local diffeomorphism $F$ and hence $\mathcal{L}_{\Gamma}d^{\mathcal{L}^{\tau}E}F^{*}\lambda_{E}=d^{\mathcal{L}^{\tau}E}\Theta_{\Gamma,F}$. Then the equations
$$
\mathcal{L}_{\Gamma}d^{\mathcal{L}^{\tau}E}F^{*}\lambda_{E}(H_{\eta},H_{\beta})=0, \quad \mathcal{L}_{\Gamma}d^{\mathcal{L}^{\tau}E}F^{*}\lambda_{E}(H_{\eta},\tilde{V}_{\beta})=0 \mbox{ and }  \mathcal{L}_{\Gamma}d^{\mathcal{L}^{\tau}E}F^{*}\lambda_{E}(\tilde{V}_{\eta},\tilde{V}_{\beta})=0,
$$
together with
$$
\mathcal{L}_{\Gamma}d^{\mathcal{L}^{\tau}E}F^{*}\lambda_{E}(H_{\eta},\tilde{V}_{\beta})-\mathcal{L}_{\Gamma}d^{\mathcal{L}^{\tau}E}F^{*}\lambda_{E}(H_{\beta},\tilde{V}_{\eta})=0
$$
yield the Helmholtz conditions given in \cite{POP}.

In order to check this we first compute $[\Gamma,\tilde{V}_{\eta}]$ and $[\Gamma,H_{\eta}]$ in terms of the basis $\{H_{\alpha},\tilde{V}_{\alpha}\}$:

$$
[\Gamma,\tilde{V}_{\eta}]=-\tilde{T}_{\eta}-\frac{\partial\Gamma^{\alpha}}{\partial y^{\eta}}\tilde{V}_{\alpha}=-(H_{\eta}-\Lambda_{\eta}^{\beta}\tilde{V}_{\beta})-\frac{\partial\Gamma^{\alpha}}{\partial y^{\eta}}\tilde{V}_{\alpha}=-H_{\eta}+\frac{1}{2}\left(C^{\gamma}_{\beta\eta}y^{\beta}-\frac{\partial \Gamma^{\gamma}}{\partial y^{\eta}}\right)\tilde{V}_{\gamma},
$$

\begin{eqnarray*}
[\Gamma,H_{\eta}]&=&[\Gamma,\tilde{T}_{\eta}]+[\Gamma,\Lambda^{\gamma}_{\eta}\tilde{V}_{\gamma}]=[\Gamma,\tilde{T}_{\eta}]+\rho(\Gamma)(\Lambda^{\gamma}_{\eta})\tilde{V}_{\gamma}+\Lambda^{\gamma}_{\eta}[\Gamma,\tilde{V}_{\gamma}]\\
&=& -\left( \rho^{i}_{\eta}\frac{\partial\Gamma^{\alpha}}{\partial x^{i}}\tilde{V}_{\alpha}+y^{\alpha}C^{\gamma}_{\eta\alpha}(H_{\gamma}-\Lambda^{\nu}_{\gamma}\tilde{V}_{\nu}) \right) + \left(y^{\alpha}\rho^{i}_{\alpha}\frac{\partial\Lambda^{\gamma}_{\eta}}{\partial x^{i}}+\Gamma^{\alpha}\frac{\partial\Lambda^{\gamma}_{\eta}}{\partial y^{\alpha}} \right)\tilde{V}_{\gamma}\\
& & + \Lambda^{\gamma}_{\eta}\left(-H_{\gamma}+\frac{1}{2}\left( C^{\nu}_{\beta\gamma}y^{\beta}-\frac{\partial\Gamma^{\nu}}{\partial y^{\gamma}} \right)\tilde{V}_{\nu}\right) = \frac{1}{2}\left(y^{\beta}C_{\beta\alpha}^{\gamma}-\frac{\partial \Gamma^{\gamma}}{\partial y^{\alpha}} \right)H_{\gamma} \\
& &+\left( y^{\alpha}\rho^{i}_{\alpha}\frac{\partial\Lambda^{\gamma}_{\eta}}{\partial x^{i}}+\Gamma^{\alpha}\frac{\partial\Lambda^{\gamma}_{\eta}}{\partial y^{\alpha}}+\Lambda^{\nu}_{\eta}\Lambda^{\gamma}_{\nu}-\Lambda^{\nu}_{\eta}\frac{\partial\Gamma^{\gamma}}{\partial y^{\nu}}-\rho^{i}_{\eta}\frac{\partial\Gamma^{\gamma}}{\partial x^{i}}+y^{\alpha}C^{\nu}_{\eta\alpha}\Lambda^{\gamma}_{\nu} \right)\tilde{V}_{\gamma}.
\end{eqnarray*}

We introduce the notation 
$$
D_{\eta}^{\gamma}:=\frac{1}{2}\left(y^{\beta}C_{\beta\alpha}^{\gamma}-\frac{\partial \Gamma^{\gamma}}{\partial y^{\alpha}} \right) \mbox{ and } \Phi^{\gamma}_{\eta}:=\left( y^{\alpha}\rho^{i}_{\alpha}\frac{\partial\Lambda^{\gamma}_{\eta}}{\partial x^{i}}+\Gamma^{\alpha}\frac{\partial\Lambda^{\gamma}_{\eta}}{\partial y^{\alpha}}+\Lambda^{\nu}_{\eta}\Lambda^{\gamma}_{\nu}-\Lambda^{\nu}_{\eta}\frac{\partial\Gamma^{\gamma}}{\partial y^{\nu}}-\rho^{i}_{\eta}\frac{\partial\Gamma^{\gamma}}{\partial x^{i}}+y^{\alpha}C^{\nu}_{\eta\alpha}\Lambda^{\gamma}_{\nu} \right)
$$
so that $[\Gamma,\tilde{V}_{\eta}]=-H_{\eta}+D^{\gamma}_{\eta}\tilde{V}_{\gamma}$ and $[\Gamma,H_{\eta}]=D_{\eta}^{\gamma}H_{\gamma}+\Phi^{\gamma}_{\eta}\tilde{V}_{\gamma}$.

We will also need the expression of $d^{\mathcal{L}^{\tau}E}F^{*}\lambda_{E}$ in terms of $\{\theta^{\alpha}:=\tilde{V}^{\alpha}-\Lambda_{\beta}^{\alpha}\tilde{T}^{\beta},\tilde{T}^{\alpha}\}$, the dual basis of $\{H_{\alpha},\tilde{V}_{\alpha}\}$:
$$
d^{\mathcal{L}^{\tau}E}F^{*}\lambda_{E}=\left(\rho(H_{\gamma})(F_{\alpha})-\frac{1}{2}F_{\nu}C^{\nu}_{\gamma\alpha} \right)\tilde{T}^{\gamma}\wedge \tilde{T}^{\alpha}+\frac{\partial F_{\alpha}}{\partial y^{\gamma}}\theta^{\gamma}\wedge \tilde{T}^{\alpha}=A_{\gamma\alpha}\tilde{T}^{\gamma}\wedge \tilde{T}^{\alpha}+\frac{\partial F_{\alpha}}{\partial y^{\gamma}}\theta^{\gamma}\wedge \tilde{T}^{\alpha}\,
$$
where $A_{\gamma\alpha}=\rho(H_{\gamma})(F_{\alpha})-\frac{1}{2}F_{\nu}C^{\nu}_{\gamma\alpha}$.

Now we introduce the notation $T_{F}:=d^{\mathcal{L}^{\tau}E}F^{*}\lambda_{E}$ and write the Helmholtz conditions in local coordinates as follows: 
\begin{eqnarray}
\mathcal{L}_{\Gamma}T_{F}(H_{\eta},H_{\beta})&=&\Gamma(T_{F}(H_{\eta},H_{\beta}))-T_{F}([\Gamma,H_{\eta}],H_{\beta})-T_{F}(H_{\eta},[\Gamma,H_{\beta}]) \nonumber \\
& = &\Gamma(A_{\eta\beta})-\Gamma(A_{\beta\eta})-\left[A_{\gamma\beta}D^{\gamma}_{\eta}-A_{\beta\gamma}D^{\gamma}_{\eta}+\frac{\partial F_{\beta}}{\partial y^{\gamma}}\Phi^{\gamma}_{\eta} \right] \nonumber \\
& & -\left[A_{\eta\gamma}D^{\gamma}_{\beta}-A_{\gamma\eta}D^{\gamma}_{\beta}-\frac{\partial F_{\eta}}{\partial y^{\gamma}}\Phi^{\gamma}_{\beta} \right] = 0 \label{H12}
\end{eqnarray}

\begin{eqnarray}
\mathcal{L}_{\Gamma}T_{F}(H_{\eta},\tilde{V}_{\beta})&=&\Gamma(T_{F}(H_{\eta},\tilde{V}_{\beta}))-T_{F}([\Gamma,H_{\eta}],\tilde{V}_{\beta})-T_{F}(H_{\eta},[\Gamma,\tilde{V}_{\beta}]) \nonumber \\
&=&\Gamma\left(-\frac{\partial F_{\eta}}{\partial y^{\beta}}\right)+\frac{\partial F_{\gamma}}{\partial y^{\beta}}D^{\gamma}_{\eta}-\left[-A_{\eta\beta}+A_{\beta\eta}-\frac{\partial F_{\eta}}{\partial y^{\gamma}}D^{\gamma}_{\beta} \right]=0 \label{H22}
\end{eqnarray}

\begin{eqnarray}
\mathcal{L}_{\Gamma}T_{F}(\tilde{V}_{\eta},\tilde{V}_{\beta})&=& -T_{F}([\Gamma,\tilde{V}_{\eta}],\tilde{V}_{\beta})-T_{F}(\tilde{V}_{\eta},[\Gamma,\tilde{V}_{\beta}]) \nonumber \\
&=& T_{F}(H_{\eta},\tilde{V}_{\beta})+T_{F}(\tilde{V}_{\eta},H_{\beta})=-\frac{\partial F_{\eta}}{\partial y^{\beta}}+\frac{\partial F_{\beta}}{\partial y^{\eta}}=0 \label{cross}
\end{eqnarray}

Now we compute
\begin{eqnarray*}
\mathcal{L}_{\Gamma}T_{F}(H_{\eta},\tilde{V}_{\beta})-\mathcal{L}_{\Gamma}T_{F}(H_{\beta},\tilde{V}_{\eta}) & = & \Gamma\left( -\frac{\partial F_{\eta}}{\partial y^{\beta}} \right)+\frac{\partial F_{\gamma}}{\partial y^{\beta}}D^{\gamma}_{\eta}-\left[-A_{\eta\beta}+A_{\beta\eta}-\frac{\partial F_{\eta}}{\partial y^{\gamma}}D^{\gamma}_{\beta}  \right] \nonumber \\
& & -\Gamma\left( -\frac{\partial F_{\beta}}{\partial y^{\eta}} \right) -\frac{\partial F_{\gamma}}{\partial y^{\eta}}D^{\gamma}_{\beta} +\left[ -A_{\beta\eta}+A_{\eta\beta}-\frac{\partial F_{\beta}}{\partial y^{\gamma}}D^{\gamma}_{\eta} \right] 
\end{eqnarray*}
and use (\ref{cross}) to obtain
\begin{equation}
A_{\eta\beta}=A_{\beta\eta}. \label{symm}
\end{equation} 

Substituting (\ref{symm}) into (\ref{H12}) and (\ref{H22}) these equations become
$$
\frac{\partial F_{\beta}}{\partial y^{\gamma}}\Phi^{\gamma}_{\eta}=\frac{\partial F_{\eta}}{\partial y^{\gamma}}\Phi^{\gamma}_{\beta} \quad \mbox{ and } \quad \Gamma\left(\frac{\partial F_{\eta}}{\partial y^{\beta}}\right)-\frac{\partial F_{\gamma}}{\partial y^{\beta}}D^{\gamma}_{\eta}-\frac{\partial F_{\eta}}{\partial y^{\gamma}}D^{\gamma}_{\beta}=0,
$$
which are the equations given in \cite{POP}. This can be checked directly by making the substitution
$$
\Lambda^{\nu}_{\eta}\frac{\partial\Gamma^{\gamma}}{\partial y^{\nu}}=2\Lambda^{\nu}_{\eta}\Lambda^{\gamma}_{\nu}+\Lambda^{\nu}_{\eta}C^{\gamma}_{\nu\tau}y^{\tau}.
$$
Beware that the notation in \cite{POP} is $N^{\gamma}_{\eta}=-\Lambda^{\gamma}_{\eta}$.

Note that the Helmholtz conditions for invariant Lagrangians on the tangent bundle of a Lie group $G$ given in \cite{CM2008} are also recovered. Indeed, by dropping the terms where derivatives with respect to $x^{i}$ appear and substituting $\Lambda^{\gamma}_{\eta}=\frac{\partial \Gamma^{\gamma}}{\partial y^{\eta}}-C^{\gamma}_{\eta\beta}y^{\beta}$ we get
$$
\Phi^{\gamma}_{\eta}=\frac{1}{2}\Gamma^{\alpha}\frac{\partial^{2}\Gamma^{\gamma}}{\partial y^{\alpha}\partial y^{\eta}}-\frac{1}{2}\Gamma^{\alpha}C^{\gamma}_{\eta\alpha}-\frac{1}{4}\frac{\partial \Gamma^{\nu}}{\partial y^{\eta}}\frac{\partial \Gamma^{\gamma}}{\partial y^{\nu}}-\frac{1}{4}C^{\nu}_{\eta\beta}y^{\beta}C^{\gamma}_{\nu\tau}y^{\tau}-\frac{1}{4}\frac{\partial \Gamma^{\nu}}{\partial y^{\eta}}C^{\gamma}_{\nu\tau}y^{\tau}+\frac{3}{4}C^{\nu}_{\eta\beta}y^{\beta}\frac{\partial\Gamma^{\gamma}}{\partial y^{\nu}}.
$$


\section*{Acknowledgments}

This work has been partially supported by MINECO (Spain) MTM2013-42870-P and MICINN (Spain)
MTM2010-21186-C02-01, MTM2010-21186-C02-02; 2009SGR1338 from the Catalan government; the European project IRSES-project GeoMech-246981 and the ICMAT Severo Ochoa project SEV-2011-0087.
MFP has been financially supported by a FPU scholarship from MECD.

\bibliographystyle{plain}
\bibliography{References}
\end{document}